\documentclass[reqno]{amsart} 
\usepackage{amssymb, amstext, amscd, amsmath}
\usepackage{geometry}
\usepackage{relsize}
\usepackage{hyperref}

\newcommand{\bC}{\mathbb{C}}

\newcommand{\bZ}{\mathbb{Z}}

\newtheorem{satz}{Satz}
\newtheorem{lemma}[satz]{Lemma}
\newtheorem{proposition}[satz]{Proposition}
\newtheorem{theorem}[satz]{Theorem}
\newtheorem{corollary}[satz]{Corollary}
\newtheorem{definition}[satz]{Definition}

\theoremstyle{remark}
\newtheorem{example}[satz]{Example}
\newtheorem{remark}[satz]{Remark}
\newtheorem{notation}[satz]{Notation}
\begin{document}

\title[Noncommutative Chern form on \'etale groupoid is closed]
{Noncommutative Chern form on \'etale groupoid is closed}

\author[W. Zhang]{Wen Zhang}
\address{Department of Mathematics\\Jilin University\\Changchun 130012\\P. R. China}
\email{wenzhang19@mails.jlu.edu.cn}

\begin{abstract}
We use bisection to provide an algebraic proof that the Chern form on the convolution algebra of an \'etale groupoid is closed.
\end{abstract}


\maketitle

%%%%%%%%%%%%%%%%%%%%%%%%%%%%%%

\section{Introduction}
 In this paper, we give a new proof that the Chern form in \cite[(6.43)]{gorokhovsky2003local} is closed.
Gorokhovsky and Lott provided a superconnection heat kernel proof, in the style of Bismut, of Connes' index theorem (cf. \cite[Theorem 6]{gorokhovsky2003local}). They considered a smooth \'etale groupoid $G$ acting on a $G$-proper manifold and a $G$-Dirac type operator $D$. As pointed out in \cite{LeichtnamPiazza+2005+169+233}, Connes' index theorem for $G$-proper manifold unified most of the existing index theorems at that time under a single statement. We first review index theory and then explain Gorokhovsky and Lott's work.

As a starting point, we recall the family index theorem. The family index theorem describes a family of elliptic pseudodifferential operators that depend continuously on a parameter from some compact space. The index of a family of operators is an element in the 
$K$-theory of the base space. A special case occurs when the kernel and the cokernel are vector bundles. In such a case, the index is the difference of the classes of these bundles (\cite{atiyah1971index}).

One may reformulate this situation by replacing the base space $B$ by the $C^*$-algebra $C(B)$ and the family of operators by an operator on a $C(B)$-module. The index is then in the $C^*$-algebraic $K$-theory $K_0(C(B))$, which is naturally isomorphic to $K^0(B)$. For general $C^*$-algebras, the Mishchenko-Fomenko index theorem (\cite{mishchenko1979index}) for elliptic operators on Hilbert $C^*$-modules formalized and generalized this point of view. The index of such an elliptic operator is an element in the $K$-theory of the $C^*$-algebra.

The index problem is simplified by considering the Chern character of $Ind(D)$, which lies in de Rham cohomology $H^*(B)$. It can be proven that $\operatorname{Ch}(Ind(D))$ equals the cohomology class of the superconnection Chern form (cf. \cite[Theorem 9.33]{berline2003heat}). Then Bismut's family index theorem follows from a canonical proof using superconnection formalism.
For more details, see \cite{berline2003heat, MR2344015}.

Gorokhovsky and Lott (\cite{gorokhovsky2003local}) extended the local family index theorem to the convolution algebra of an \'etale groupoid.
According to Connes' index theorem, the index of $D$ belongs to the $K$-theory group $K_0(C_r^*(G))$. Since $C_r^*(G)$ lacks a naturally dense smooth subalgebra stable under the holomorphic functional calculus, Gorokhovsky and Lott overcome this problem by defining $Ind(D)$ as the $K$-theory group element represented by the difference between the index projection and a standard projection. In their setting, $Ind(D)$ is an element in the $K$-theory of a certain algebra of smoothing $C^*(G)$-operators. 

 They developed homology computations and defined a graded differential algebra (GDA, for short) that can be considered a space of noncommutative forms in the general \'etale groupoid case, whose cohomology generalized the de Rham cohomology.
They also constructed the noncommutative connection, traces, and the Chern character 
 \begin{equation}\label{2024.07.02}
\operatorname{Ch}: K_0(End_{C_c^\infty(G)}^\infty(\Gamma_c^\infty(P; E))) \longrightarrow \Omega_c^{\bullet}(G)_{\rm{Ab}}.
\end{equation}
necessary for a noncommutative local index theory.

The relationship between the Chern character of the index (paired with a suitable closed graded trace) and the superconnection Chern character (paired with the same trace) is explained in \cite[Section 5]{gorokhovsky2003local}. See also \cite[Section 8]{LeichtnamPiazza+2005+169+233}, this is a key step in finishing the proof of the index theorem.
Gorokhovsky and Lott claimed that the superconnection Chern form is closed, and this can be proved using similar arguments as in Heitsch's paper (\cite{MR1361580}). However, the proof presented in \cite{MR1361580} relies on many analytical techniques that are not very clear in terms of their extension to noncommutative cases.
 On the other hand, in the crossed product case, Leichtnam and Piazza (\cite[Proposition 12]{LeichtnamPiazza+2005+169+233}) provide a detailed algebraic proof parallel to the classical proof of \cite[Chapter 9]{berline2003heat}.

In this paper, we mainly provide an algebraic proof that the superconnection Chern form in \cite[(6.43)]{gorokhovsky2003local} is a closed form in the \'etale groupoid case. As Leichtnam and Piazza (\cite{LeichtnamPiazza+2005+169+233}) write, such a result is interesting and reassuring in its own right.
Our motivation comes from studying the extension of
\cite{so2017non, LeichtnamPiazza+2005+169+233} from \'etale crossed product groupoids to general \'etale groupoids.
This extension enables us to define a noncommutative analytic torsion form on the convolution algebra of an \'etale groupoid. The noncommutative analytic torsion form will be further explored in an upcoming paper.

Recall that the crossed product groupoid is the set $B\times \Gamma$. Using this fact, the authors of the above papers decompose any elements in $\Omega_c^{\bullet}(B\rtimes \Gamma)$ as a sum $\sum_{d g_{(k)} g} \omega^{d g_{(k)} g} d g_{(k)} g$ and then perform explicit computations.  In the \'etale case, 
we find a similar decomposition using bisections and demonstrate that similar computations hold.

The rest of this paper is organized as follows.
 In Sec. \ref{Setup}, we provide some background information.
 We consider a fiber bundle $P\rightarrow G^{(0)}$ and a vector bundle over $P$. We introduce a $G$ action on these bundles and a $G$-invariant metric. Additionally, we explain the Bismut bundle and define superconnections.
 
In Sec. \ref{Convolution Algebra},
We first define the convolution algebra and present 
the bisection method along with related results.
Then we generalize the definition of vector representation, which leads to a module structure. 
We use the groupoid of germs, a special feature of \'etale groupoids, to define a $G$-invariant connection and extend it to a noncommutative connection, along with its Chern character. 

In Sec. \ref{Noncommutative connection and trace on a Bismut bundle}, 
we extend Sec. \ref{Convolution Algebra} to the Bismut bundle case and clarify \cite{gorokhovsky2003local} by writing down explicit formulas.
We rewrite a class of $\Omega_c^{\bullet}(G)$-linear smoothing operators to facilitate direct calculations and derive their necessary conditions.
Then we present the main results of this paper.

In Sec. \ref{Proof of Theorem}, we prove Theorem \ref{2024.03.04},
which states
$$(d_1+d_2)(Tr_{<e>}(K))=Tr_{<e>}([D_B+\nabla^{0,1}, K]).$$
This theorem extends \cite[Lemma 3.19]{so2017non} and
\cite[Lemma 14]{LeichtnamPiazza+2005+169+233} to the \'{e}tale case, thereby 
constituting a generalization of \cite[Lemma 9.15]{berline2003heat}.

Our findings help clarify the translation from étale crossed product groupoids to general étale groupoids, enabling direct operations on the latter.

\section{Setup}\label{Setup}
This section will first focus on \'etale groupoids and then introduce the geometric settings from \cite{gorokhovsky2003local, so2017non}.
\subsection{\'Etale groupoid acting on a vector bundle}
\begin{definition}
A topological groupoid $G$ is called an \'etale groupoid if the target map $t: G\rightarrow G^{(0)}$ is a local homeomorphism.
\end{definition}

Let $G$ be a smooth Hausdorff proper \'etale groupoid over $G^{(0)}$ with source map s and target map t. For $x\in G^{(0)}$, let $G^x=t^{-1}(x)$ and $G_x=s^{-1}(x)$.
Let $Z\rightarrow P\stackrel{\pi}{\rightarrow} G^{(0)}$ be a smooth fiber bundle
and let $P\times_{G^{(0)}}G=\{(p, \gamma):\pi(p)=t(\gamma)\}$.
Given $x\in G^{(0)}$, we write $Z_x=\pi^{-1}(x)$.
A groupoid can act on a space fibered over its base with a fiber-respecting map, i.e., $\gamma: Z_{t(\gamma)}\rightarrow Z_{s(\gamma)}$.
\begin{definition}
A right action of $G$ on P along the map $\pi$, which is called the moment map, is a mapping $(p, \gamma)\mapsto p\gamma$ from $P\times_{G^{(0)}}G$ to $P$ satisfying:
\renewcommand{\labelenumi}{(\roman{enumi})}
\begin{enumerate}
\item $\pi(p\gamma)=s(\gamma)$,
\item $p(\gamma\gamma_1)=(p\gamma)\gamma_1$,
\item $p1_{\pi(p)}=p$.
\end{enumerate}
\end{definition}
It follows that the map $p\rightarrow p\gamma$ gives a diffeomorphism from $Z_{t(\gamma)}$ to $Z_{s(\gamma)}$ for each $\gamma\in G$. We assume that the submersion $\pi$ is $G$-equivariant, i.e., $\pi(p\gamma)=\pi(p)\gamma$.

Following 
\cite{gorokhovsky2003local},
we assume that $P$ is a proper $G$-manifold and the groupoid $G$ acts freely and cocompactly on $P$.
Then the groupoid $\mathcal{G}=P \rtimes G$ with underlying space $P \times_{G^{(0)}} G$, with space of units $\mathcal{G}^{(0)}=P$ and maps $t(p, \gamma)=p$ and $s(p, \gamma)=p \gamma$, is a proper, free and cocompact groupoid.
The quotient space $M=P/G$ is thus a compact manifold.
We assume $P$ has a $G$-invariant fiberwise Riemannian metric.
Denote the $G$-invariant Riemannian densities on the fibers $\{Z_x\}_{x\in G^{(0)}}$ by 
$\{\mu_x\}_{x\in G^{(0)}}$. 
\begin{definition}
A complex vector bundle $E\stackrel{\wp}{\rightarrow} P$ is called a $G$-equivariant vector bundle over $P$ if $G$ also acts on E from the right with moment $\pi\circ{\wp}$, such that for any $(e, \gamma)\in E\times_{G^{(0)}}G$, $\wp(e\gamma)=\wp(e)\gamma$.
\end{definition}
Given a $G$-equivariant vector bundle $E$, $E\stackrel{\wp}{\rightarrow} P$ is a representation of $\mathcal G$ according to \cite[Definition 1.16]{mestre2016differentiable}, i.e., a $\mathcal G$-equivariant vector bundle over $\mathcal G^{(0)}=P$ by defining a map $(e, (p, \gamma))\mapsto e\gamma$ from $E\times_{\mathcal G^{(0)}}\mathcal G$ to $E$. 
Thus by \cite[Proposition 1.43]{mestre2016differentiable}, there exists a $G$-invariant inner product $\langle\cdot, \cdot\rangle$ on $E$, i.e., a family of inner products on the fibers of $E$ such that 
$$\langle e_1\gamma, e_2\gamma\rangle=\langle e_1, e_2\rangle,$$
whenever $e_1$ and $e_2$ are in the same fiber $E_{\wp(e_1)}$.
\begin{example}
Let $E'$ be the dual bundle of $E$ with $\xi_{e_1}(e_2):=\langle e_1, e_2\rangle$, where $\xi_{e_1}\in E'_p, e_1, e_2\in E_p$ for any $p\in P$. A groupoid action on $E'$ from the right is given by
$$(\xi_{e_1}\gamma)(e_2)
=\xi_{e_1}(e_2\gamma^{-1}).$$
Hence $E'$ is also a $G$-equivariant vector bundle.
\end{example}

\subsection{The Bismut bundle}
We regard $\Gamma_c^{\infty}(P; E)$ as the space of sections of an ``infinite dimensional bundle'' over $G^{(0)}$. We call this ``bundle" a Bismut bundle $E_{\mathfrak{b}}$. Roughly speaking, each section $F\in\Gamma_c^{\infty}(P; E)$ is regarded as a map 
$$x\longmapsto F|_{Z_x}\in\Gamma_c^{\infty}(Z_x; E|_{Z_x})\ \mathrm{ for\ all }\ x\in G^{(0)}.$$
In other words, one defines a section on $E_{\mathfrak{b}}$ to be smooth, if the images of all $x\in G^{(0)}$ fit together to form an element in $\Gamma_c^{\infty}(P; E)$,
and one only needs to consider $F\in\Gamma_c^{\infty}(P; E)$.

Let $ V_P:=\mathrm {Ker}(d\pi)$ denote the vertical tangent bundle. 
We need a grading $E^{\bullet}:=E\otimes\wedge^\bullet V_P'$.

The following definition essentially appears in \cite[Proposition 10.1]{berline2003heat}, which is the starting point of the Bismut superconnection formalism (cf. \cite{berline2003heat}).
\begin{definition}\label{2022.9.20}
On $\Omega(P; E^\bullet)$, we denote the exterior differential operator by
$D_B$, which can be split into three terms 
\[D_B=d_V^{\nabla^E}+L^{E_{\mathfrak{b}}^\bullet}+\iota_\theta,
\]
according to horizontal and vertical degrees.
For each term's explicit formula, one can see \cite[Page 324]{bismut1995flat}.
\end{definition}
Note that the first and third terms of $D_B$ are fiberwise operators, i.e., they commute with functions on the base space. The second term $L^{E_{\mathfrak{b}}^\bullet}$ is a connection.
Following \cite[(3.29)]{bismut1995flat},
the standard metric on $\Gamma_c^{\infty}(E_{\mathfrak{b}})$ is defined by 
\begin{equation}\label{2022.2.28}
\langle u_1, u_2\rangle_{E_{\mathfrak{b}}}\big(x\big):=
\int_{Z_{x}}\langle u_1(p), u_2(p)\rangle_E~\mu_{x}(p).
\end{equation}
An adjoint superconnection $D_B'$ is defined by
\begin{equation}
d_{G^{(0)}}\langle u_1, u_2\rangle_{E_{\mathfrak{b}}}=\langle D_Bu_1, u_2\rangle_{E_{\mathfrak{b}}}-\langle u_1, D_B'u_2\rangle_{E_{\mathfrak{b}}}.
\end{equation}
Explicitly (cf. \cite[Proposition 3.7]{bismut1995flat}),
\[D_B'=(d_V^{\nabla^E})^*+(L^{E_{\mathfrak{b}}^\bullet})'-\theta\wedge.
\]
From \cite[Remark 3.10]{bismut1995flat}, 
we know that $\frac{1}{2}(D_B+D_B')$ is a Bismut superconnection defined in \cite[10.3]{berline2003heat}. Therefore the results in this paper are compatible with \cite{gorokhovsky2003local}.

\section{Convolution algebra}\label{Convolution Algebra}
We briefly recall the definition of Gorokovsky and Lott's algebra-GDA-module structure. In addition, we obtain a star GDA by defining a suitable involution. 
Then, we review the definitions of noncommutative connection and Chern character. 
\begin{definition}
Let $G\rightrightarrows G^{(0)}$ be an \'etale groupoid. $C_c^{\infty}(G)$ is the natural convolution algebra of smooth functions of compact support on G.
And $C_c^{\infty}(G)$ equipped with multiplication and involution:
\[\begin{split}(f_1\star f_2)(\gamma_0):=&
\sum_{\gamma\gamma'=\gamma_0}
f_1(\gamma)f_2(\gamma'),\\
f^*(\gamma):=&\overline {f(\gamma^{-1})}.
\end{split}\]
\end{definition}

Let $G^{(n)}:=\{(\gamma_1, \cdots, \gamma_{n})\in G^n: t(\gamma_{i+1})=s(\gamma_{i}), i=1, \cdots, n-1\}$. It is a manifold of the same dimension as G. Let
$\Omega_c^{m, n}(G)$ be the quotient of $\Omega_c^m(G^{(n+1)})$ by 
the forms which are supported on
$\{(\gamma_0, \cdots, \gamma_n): \gamma_j~\mathrm{is~a~unit~ for ~some}~j>0\}.$
We define $t_n: G^{(n)}\rightarrow G^{(0)}$ and $s_n: G^{(n)}\rightarrow G^{(0)}$ by $t_n(\gamma_1, \cdots, \gamma_{n})=t(\gamma_1)$ and 
$s_n(\gamma_1, \cdots, \gamma_{n})=s(\gamma_n)$, respectively.
They are also local homeomorphisms and induce isomorphisms between tangent spaces.

The following definition is parallel to that of \cite[Definition 3.4]{so2017non}, see also \cite[(6.3)]{gorokhovsky2003local} and \cite[Section 4]{lott1999diffeomorphisms}.
\begin{definition}
The non-commutative deRham differential form is the vector space $\Omega_c^{m, n}(G)$
equipped with multiplication and involution. If $\omega_1\in \Omega_c^{m_1, k}(G)$ and $\omega_2\in \Omega_c^{m_2, l}(G)$, then
$\omega_1\star\omega_2\in \Omega_c^{m_1+m_2, k+l}(G)$ is defined by
\begin{equation}\begin{split}\label{2022.10.3}
(\omega_1\star\omega_2)(\gamma_0; \cdots, \gamma_{k+l}):=
&\sum_{\gamma\gamma'=\gamma_k}\omega_1(\gamma_0; \cdots, \gamma)\wedge\omega_2(\gamma'; \cdots, \gamma_{k+l})-\\
&\sum_{\gamma\gamma'=\gamma_{k-1}}\omega_1(\gamma_0; \cdots, \gamma, \gamma')\wedge\omega_2(\gamma_k; \cdots, \gamma_{k+l})+\cdots+\\
(-1)^k&\sum_{\gamma\gamma'=\gamma_{0}}\omega_1(\gamma; \gamma',\cdots, \gamma_{k-1})\wedge\omega_2(\gamma_k; \cdots, \gamma_{k+l}).\\
\end{split}
\end{equation}
In forming the wedge product, the pullbacks of $s_{k+1}$ and $t_{l+1}$ are used to identify cotangent spaces.
And
\[\begin{split}
\omega^*(\gamma_0; \cdots, \gamma_k):=
(-1)^k\bigg(&\sum_{\gamma\gamma'=\gamma_k^{-1}}\overline\omega(\gamma; \gamma', \cdots, \gamma_1^{-1})\chi_{G^{(0)}}(\gamma_0^{-1})-\cdots+\\
(-1)^{k-1}&\sum_{\gamma\gamma'=\gamma_1^{-1}}\overline\omega(\gamma_k^{-1}; \cdots, \gamma, \gamma')\chi_{G^{(0)}}(\gamma_0^{-1})+\\
(-1)^{k}&\sum_{\gamma\gamma'=\gamma_0^{-1}}
\overline\omega(\gamma_k^{-1}; \cdots, \gamma_1^{-1}, \gamma)\chi_{G^{(0)}}(\gamma')\bigg).\\
\end{split}\]
\end{definition}
\begin{lemma}The involution defined above satisfy
\[\begin{split}
\omega^{**}(\gamma_0; \cdots, \gamma_k)&=\omega(\gamma_0; \cdots, \gamma_k),\\
(\omega_1\star\omega_2)^*(\gamma_0; \cdots, \gamma_{k+l})&=
(\omega_2^*\star\omega_1^*)(\gamma_0; \cdots, \gamma_{k+l}).
\end{split}\]
\begin{proof}The details of the proof are omitted.
\end{proof}
\end{lemma}

Let $d_1$ be the de Rham differential on $\Omega_c^{*, *}(G)$ and define $d_2: \Omega_c^{*, *}(G)\rightarrow \Omega_c^{*, *+1}(G)$
\begin{equation}
(d_2\omega)(\gamma_0; \gamma_1,\cdots, \gamma_{n}):=
\chi_{G^{(0)}}(\gamma_0)\omega(\gamma_1; \cdots, \gamma_{n}).
\end{equation}
From the total complex of $\Omega_c^{*, *}(G)$, we obtain a star GDA. 
\begin{definition}
Let $\Omega_c^{\bullet}(G):=\bigoplus\limits_{m, n\geq0}\Omega^{m, n}_c(G)$ denote the star GDA. Let
\[[\Omega_c^{\bullet}(G), \Omega_c^{\bullet}(G)]\subseteq \Omega_c^{\bullet}(G)
\]
be the subspace spanned by graded commutators and define
\[\Omega_c^{\bullet}(G)_{\rm{Ab}}:=
\Omega_c^{\bullet}(G)\left/\overline{[\Omega_c^{\bullet}(G), \Omega_c^{\bullet}(G)]}.\right.
\]
\end{definition}
Note that the over-line denotes the closure (in the point-wise convergence topology).
By the derivation property, the differential $d_1+d_2$ perserves $[\Omega_c^{\bullet}(G), \Omega_c^{\bullet}(G)]$. 
Then $d_1+d_2$ descends to $\Omega_c^{\bullet}(G)_{\rm{Ab}}$ with total degree 1.
For complexes $(\Omega_c^{\bullet}(G)_{\rm{Ab}}, d_1+d_2)$, 
we use $H^{\bullet}(\Omega_c^{\bullet}(G)_{\rm{Ab}})$ 
denote its cohomology.

\subsection{Bisection}\label{Bisection}
In the following, we briefly review the theory of bisection based on \cite{exel2008inverse} and \cite{sims2020operator}, not only for the reader's convenience but also because we will present a few improvements.

Let $U$ and $V$ denote open subsets in $G$.
Let $U^{-1}=\{u^{-1}: u\in U\}$ and $UV=\{uv: u\in U, v\in V, (u, v)\in G^{(2)}\}$(possibly empty).
\begin{definition}
An open subset $U\subseteq G$ is called a bisection if the restrictions of s and t to $U$ are injective.
\end{definition}
As pointed out in \cite{exel2008inverse}, 
an \'etale groupoid can be defined by its 
bisections (called slices in that paper).
\begin{proposition}\label{2024.4.8}
We list some basic facts about bisection:
\renewcommand{\labelenumi}{(\roman{enumi})}
\begin{enumerate}
\item \label{2023.3.31} G has a basis of bisections (cf. \cite[Proposition 3.5]{exel2008inverse}, \cite[Lemma 8.4.9]{sims2020operator}),
\item \label{2023.3.22}If $U$ and $V$ are bisections, then $U^{-1}$ and $UV$ are bisections (cf. \cite[Proposition 3.8]{exel2008inverse}),
\item For $f_1, f_2\in C_c^{\infty}(G)$, if $supp(f_1)\subseteq U$ and $supp(f_2)\subseteq V$, then $supp(f_1^*)\subseteq U^{-1}$ and $supp(f_1\star f_2)\subseteq UV$
(cf. \cite[Lemma 9.1.4]{sims2020operator}, \cite[Proposition 3.11]{exel2008inverse}),
\item $C_c^{\infty}(G)=span\{f\in C_c^{\infty}(G)|supp(f)~~~~is~~~~in~~~~a~~~~bisection\}$ (cf. \cite[Proposition 3.10]{exel2008inverse}, \cite[Lemma 9.1.3]{sims2020operator}),
\item Define $f_1$ and $f_2$ as in $\mathrm{(iii)}$, $(f_1\star f_2)(\gamma)=f_1(\alpha)f_2(\beta)$ where $\gamma=\alpha\beta$ has a unique decomposition, which means that convolution is very easy to compute for functions supported on bisections (cf. \cite[Lemma 9.1.4]{sims2020operator}),
\item For $f_1, f_2\in C_c^{\infty}(G)$, if $supp(f_1)=supp(f_2)\subseteq U$, one get $f_1\star f_2^*\in C_c^{\infty}(G^{(0)})$ (cf. \cite[Proposition 3.12]{exel2008inverse}, \cite[Lemma 9.1.4]{sims2020operator}).
\end{enumerate}
\end{proposition}
We need a case of (vi) of Proposition \ref {2024.4.8}. This implies that one of the functions in (vi) can be chosen as a characteristic function of a bisection, although the characteristic function doesn't belong to $C_c^\infty(G)$.
\begin{corollary}\label{2024.03.31}
For $f\in C_c^{\infty}(G)$ with $supp(f)\subseteq U$, a function $1_U\in C_c^{\infty}(G)$ with $supp(1_U)\subseteq U^{-1}$ exists and equals $1$ on $t|_{U^{-1}}^{-1}s|_U(supp(f))=(supp(f))^{-1}$ such that $f\star1_U\in C_c^{\infty}(G^{(0)})$.
\end{corollary}

\begin{proof}
The existence of such a function is obvious as if $supp(1_U)$ is big enough such that 
$(supp(f))^{-1}\\
\subseteq supp(1_U)\subseteq U^{-1}$.
Observing that if $(f\star 1_U)(\gamma)\neq 0$, then there exists at least one pair $(\alpha, \beta)\in G^{(2)}$, such that $\gamma=\alpha\beta, \alpha\in U$ and $\beta\in U^{-1}$, but since $t(\alpha)=t(\gamma)$ and $t(\beta)=s(\alpha)$, we necessarily have that $\alpha=t|_U^{-1}t(\gamma)$ and $\beta=t|_{U^{-1}}^{-1}s|_U(\alpha)=\alpha^{-1}$, where we are denoting by $t|_U$ the restriction of $t$ to $U$. Hence $\gamma\in G^{(0)}$. Since $supp(f)$ is compact and $s|_U, t|_{U^{-1}}^{-1}$ are homeomorphisms, $t|_{U^{-1}}^{-1}s|_U(supp(f))$ is a compact subset of $U^{-1}$. Thus $supp(f\star 1_U)=supp(f)(supp(f))^{-1}$ is compact and $(f\star 1_U)(\gamma)=f(t|_U^{-1}t(\gamma))$
is smooth.
\end{proof}

Recall that $s_1, t_1: G^{(2)}\rightarrow G^{(0)}$ are defined by $s_1(\gamma_0, \gamma_1)=s(\gamma_1)$ and $t_1(\gamma_0, \gamma_1)=t(\gamma_0)$.
Set $U\hat\times V:=(U\times V)\cap G^{(2)}=\{(u, v): u\in U, v\in V, (u, v)\in G^{(2)}\}$(possibly empty).
\begin{definition}
An open subset $U\subseteq G^{(2)}$ is called a bisection in $G^{(2)}$ if the restrictions of $s_1
$ and $t_1$ to $U$ are injective.
\end{definition}

\begin{lemma}If $U, V\subseteq G$ are bisections, 
then $U\hat\times V\subseteq G^{(2)}$ is a bisection in $G^{(2)}$.
\end{lemma}
\begin{proof}
Since the restriction of multiplication map on $U\hat\times V\rightarrow UV$ is a continuous, open map and one-to-one, it is a homeomorphism. So $U\hat\times V$ is an open subset of $G^{(2)}$.
If $U\hat\times V$ is not empty, there exists $(\gamma_u, \gamma_v)\in U\hat\times V$ such that $t(\gamma_v)=s(\gamma_u)$. Since $V$ is a bisection, thus $\gamma_v=(t|_V)^{-1}(s(\gamma_u))$ and therefore $(\gamma_u, \gamma_v)=(\gamma_u, (t|_V)^{-1}(s(\gamma_u)))$, where we are denoting by $t|_V$ the restriction of $t$ to $V$. 

For any $(\gamma_0, (t|_V)^{-1}(s(\gamma_0)))\in U\hat\times V$ 
and $(\gamma_1, (t|_V)^{-1}(s(\gamma_1)))\in U\hat\times V$, 
suppose
$$t_1(\gamma_0, (t|_V)^{-1}(s(\gamma_0)))=t_1(\gamma_1, (t|_V)^{-1}(s(\gamma_1))),$$ 
which implies $t(\gamma_0)=t(\gamma_1)$.
Since $t|_U$ is a homeomorphism, we get $\gamma_0=\gamma_1$, it follows that $(\gamma_0, (t|_V)^{-1}(s(\gamma_0)))=(\gamma_1, (t|_V)^{-1}(s(\gamma_1)))$ and therefore $t_1|_{U\hat\times V}$ is injective. Similarly, $s_1|_{U\hat\times V}$ is injective.
\end{proof}
\begin{lemma}\label{23.3.22}
The manifold $G^{(2)}$ has a base of bisections. 
\end{lemma}
\begin{proof}
If $(\gamma_0, \gamma_1)\in G^{(2)}$, by (i) of Proposition \ref{2024.4.8}, there exist $U, V\subseteq G$ that are bisections such that $\gamma_0\in U, \gamma_1\in V$. Then $(\gamma_0, \gamma_1)$ is contained in the bisection $U\hat\times V$.
\end{proof}
Next, we extend (iv) of Proposition \ref{2024.4.8} to the case of noncommutative forms.
\begin{proposition}\label{2024.04.08}
For $\omega\in\Omega_c^{*, 1}(G)=\Omega_c^*(G^{(2)})/\sim$, we have $\omega=\sum\limits_{i}\omega_{i}$, where $supp(\omega_{i})\subseteq U_i\hat\times V_i$.
\end{proposition}
\begin{proof}
Fix $\omega\in\Omega_c^{*, 1}(G)$. By Lemma \ref{23.3.22}, we can cover $supp(\omega)$ with bisections $\{U_i\hat\times V_i: U_i, V_i\subseteq G~~are~~bisections\}$, and then use compactness to pass to a finite subcover $U_1\hat\times V_1, \cdots , U_n\hat\times V_n$. Choose a partition of unity $\{h_i\}$ on $supp(\omega)$ subordinate to the $U_i\hat\times V_i$. The pointwise products $\omega_i=\omega\cdot h_i$ belong to $\Omega_c^{*, 1}(G)$ with $supp(\omega_i)\subseteq U_i\hat\times V_i$, and we have $\omega=\sum\limits_i^n\omega_i$. \end{proof}
\begin{remark}
For convenience of later calculation, we can also write $\omega=\sum\limits_{i, j}\omega_{i, j}$, where $supp(\omega_{i, j})\subseteq U_i\hat\times V_j$. 
\end{remark}

\subsection{Vector representation}\label{Vector representation}
Let $E\rightarrow G^{(0)}$ be a $G$-equivariant vector bundle.
Thus, we can define a $C_c^\infty(G)$-module structure on $\Gamma_c^{\infty}(G^{(0)}; E)$.
\begin{definition}\label{2024.6.7}
The vector representation $\nu$ is the left action of $C_c^\infty(G)$ on $\Gamma_c^{\infty}(G^{(0)}; E)$ defined by
\begin{equation}\label{2024.5.18}
(\nu(f)F)(x)
=\sum_{\gamma\in G^x }f(\gamma)F(s(\gamma))\gamma^{-1}.
\end{equation}
\end{definition}
In Equation \eqref{2024.5.18}, we have used the fact that $\gamma^{-1}: E_{s(\gamma)}\rightarrow E_{t(\gamma)}$.
When $E$ is a trivial bundle, one can find the definition of vector representation in \cite[(6.6)]{gorokhovsky2003local} and \cite[(9)]{10.1007/978-3-0348-8364-1_8}.

We extend the vector representation to a left action of
$\Omega_c^\bullet(G)$.
From $s_n: G^{(n)}\rightarrow G^{(0)}$,
we have the pullback vector bundle over $G^{(n)}$ denoted by $s_n^*E$.
Just like in \cite[(6.12)]{gorokhovsky2003local}, there is an isomorphism 
between $\Omega_c^{m, n}(G)\otimes_{C_c^{\infty}(G)}\Gamma_c^{\infty}(G^{(0)}; E)$ and the quotient of $\Omega_c^m(G^{(n)}; s_n^*E)$ by the forms which are supported on
\[
\{(\gamma_0, \cdots, \gamma_{n-1}): \gamma_j~\mathrm{is~a~unit~for~some}~j\geq0\}.\]
It will be convenient to denote the quotient space by $\Omega_c^{m, n}(G; E)$.

In the following definition, we extend the vector representation in Definition \ref{2024.6.7} to a left action of
$\Omega_c^\bullet(G)$ on $\Omega_c^\bullet(G; E)$.
\begin{definition}\label{2024.05.18}
Let $\omega\in\Omega_c^{*, k}(G)$ and $
F\in\Omega_c^{*, l}(G; E)$. Then we define $\nu(\omega)F\in\Omega_c^{*, k+l}(G; E)$ by
\begin{equation}
\begin{split}
(\nu(\omega)F)(&\gamma_0, \cdots, \gamma_{k+l-1})\\
:=
&\mathlarger{\sum}_{i=0}^{k-1}(-1)^i\sum_{\gamma\gamma'=\gamma_i}
\omega(\gamma_0; \cdots, \gamma, \gamma', \cdots, \gamma_{k-1})\wedge
F(\gamma_k, \cdots, \gamma_{k+l-1})\\
&+(-1)^k\sum_{\gamma\gamma'=\gamma_k}
\omega(\gamma_0; \cdots, \gamma_{k-1}, \gamma)\wedge
F(\gamma',\cdots, \gamma_{k+l-1})\\
&+\mathlarger{\sum}_{j=k+1}^{k+l-1}(-1)^j\sum_{\gamma\gamma'=\gamma_j}
\omega(\gamma_0; \cdots, \gamma_{k-1}, \gamma_k)\wedge
F(\gamma_{k+1}, \cdots, \gamma, \gamma',\cdots, \gamma_{k+l-1})\\
&+(-1)^{k+l}\sum_{\gamma\gamma'=\gamma_{k+l}}
\omega(\gamma_0; \cdots, \gamma_{k-1}, \gamma_k)\wedge
F(\gamma_{k+1}, \cdots, \gamma_{k+l-1}, \gamma)\gamma^{-1}.\\
\end{split}
\end{equation}
\end{definition}

The above equation is a slight generalization of (6.6), (6.8), (6.11), and (6.12) of \cite{gorokhovsky2003local}.
The map which sends $\omega\otimes F$ to $\nu(\omega)F$ gives an isomorphism between $\Omega_c^{m', n'}(G)
\otimes_{C_c^{\infty}(G)}\Omega_c^{m, n}(G; E)$ and $\Omega_c^{m'+m, n'+n}(G; E)$.

\subsection{Noncommutative connection}
We recall that Gorokhovsky and Lott \cite[(6.14)]{gorokhovsky2003local} defined a noncommutative connection. 
\begin{definition}\label{2024.6.15}
A $\bC$-linear map $\nabla: \Gamma_c^\infty(G^{(0)}; E)\rightarrow \Omega_c^1(G)\otimes_{C_c^\infty(G)}\Gamma_c^\infty(G^{(0)}; E)$ is a connection if 
$$\nabla(\nu(f)F)=\nu(f)\nabla(F)+\nu((d_1+d_2)f)F,~\mathrm{for~any~} f\in C_c^\infty(G).$$
\end{definition}
Following \cite[Sections 4 and 5]{exel2008inverse}, all bisections of $G$ (denoted by $S(G)$) form an inverse semigroup. One can construct a groupoid of germs from an inverse semigroup action. For a bisection $U$ and $x\in t(U)$, we let
$$\theta_{U}(x)=x\cdot t|_U^{-1}(x)=s\circ t|_U^{-1}(x).$$
One gets $\theta_{U}: t(U)\rightarrow s(U)$ is a homeomorphism
and the correspondence $U\mapsto\theta_{U}$ gives an action of $S(G)$ on $G^{(0)}$ (cf. \cite[Proposition 5.3]{exel2008inverse}).
For $F\in\Gamma_c^{\infty}(G^{(0)}; E)$, each $\gamma\in G$ induces an action $[\theta_{U}]^*: \Gamma_c^{\infty}(G^{(0)}; E)\rightarrow \Gamma_c^{\infty}(t(U); E)$, where $U$ is any bisection contains $\gamma$ and  $[\theta_{U}]$ is the germ of $\theta_{U}$, explicitly, 
$$([\theta_{U}]^*F)(x)=F(x\cdot t|_U^{-1}(x))\cdot (t|_U^{-1}(x))^{-1}.$$
Using this notation, we define a $G$-invariant connection as a generalization of the
 crossed product case (cf. \cite{gorokhovsky2003local, so2017regularity, so2017non}).
\begin{definition}\label{2024.7.2}
A (commutative) connection $\nabla$ on $E$ is called $G$-invariant if 
$$\nabla([\theta_{U}]^*F)=[\theta_{U}]^*(\nabla F)$$
for any $U$ and $F$ as above.
\end{definition}  
We fix $h \in C_c^{\infty}({G}^{(0)})$ such that for all $x \in {G}^{(0)}$
$$
\sum_{\gamma \in {G}^x} h(s(\gamma))=1.
$$
Then there is a connection (cf. \cite[(6.14)]{gorokhovsky2003local})
$$
\nabla^{\text{can}}: \Gamma_c^\infty(G^{(0)}; E)\rightarrow \Omega_c^1(G)\otimes_{C_c^\infty(G)}\Gamma_c^\infty(G^{(0)}; E)
$$
of the form $\nabla^{\text{can}}=\nabla^{1,0} \oplus \nabla^{0,1}$,
where for any $F \in \Gamma_c^\infty(G^{(0)}; E), \nabla^{1,0}(F) \in \Omega_c^1({G}^{(0)}; E)$ is a $G$-invariant covariant derivative of $F$ and $\nabla^{0,1}(F) \in \Omega_c^{0,1}(G) \otimes_{C_c^{\infty}({G})} \Gamma_c^\infty(G^{(0)}; E)$ is given by
$$
(\nabla^{0,1}F)(\gamma_0)=(F(t(\gamma_0)) \gamma_0)h(s(\gamma_0)), \quad \forall \gamma_0 \notin {G}^{(0)}.
$$
Observe that 
$$(\nabla^{\text{can}})^2\in Hom_{C_c^\infty(G)}(\Gamma_c^\infty(G^{(0)}; E), \Omega_c^2(G)\otimes_{C_c^\infty(G)}\Gamma_c^\infty(G^{(0)}; E))
$$
is given by the action of a 2-form $\Theta\in\Gamma_c^\infty(G^{(0)}; \Omega_c^2(G)\otimes_{C_c^\infty(G)}End(E))$ which commutes with $C_c^\infty(G)$.

\subsection{Chern character}\label{Chern class}
Following \cite[Subsection 2.2]{lott1999diffeomorphisms} and
\cite[(6.18)]{gorokhovsky2003local}, 
we define:
\begin{definition}
The Chern character of the connection $\nabla^{\text{can}}$ on the $G$-equivariant vector bundle $E$ is defined by
\[\operatorname{Ch}(\nabla^{\text{can}})=tr(e^{-\frac{\Theta}{2\pi i}})\in \Omega_c^{\bullet}(G)_{\rm{Ab}},
\] 
where $tr$ is the pointwise trace.
\end{definition}

\section{Noncommutative connection and trace on a Bismut bundle}\label{Noncommutative connection and trace on a Bismut bundle}
In this section, we first generalize the vector representation of Definition \ref{2024.6.7} to the Bismut bundle and explain some parallel results as in Subsection \ref{Vector representation}. Then we introduce $\Omega_c^{\bullet}(G)$-linear operator, specifically providing some necessary conditions for $\Omega_c^{\bullet}(G)$-linear smoothing operator. Next, we recall how Gorokhovsky and Lott generalized the trace and the Bismut superconnection to the convolution algebra of an \'etale groupoid. Finally, we point out the main results.
\begin{definition}
The vector representation $\nu$ in the Bismut bundle case is the left action of $C_c^\infty(G)$ on $\Gamma_c^{\infty}(P; E)$ defined by
\begin{equation}\label{vr}
(\nu(f)F)(p)
=\sum_{t(\gamma)=\pi(p)}f(\gamma)F(p\gamma)\gamma^{-1}.
\end{equation}
\end{definition}
In Equation \eqref{vr}, we have used the fact that $\gamma^{-1}: E_{p\gamma}\rightarrow E_p$.

Recall \cite[(6.22)]{gorokhovsky2003local},
the authors defined $G^{(n)}\times_s P:=\{(\gamma_1, \cdots, \gamma_{n}, p)\in G^{(n)}\times P| \pi(p)=s(\gamma_{n})\}$ in order to write a fiberwise smoothing operator with differential form coefficients.
Because we use the right action instead of the left action,
we consider
$$P\times_tG^{(n)}:=
\{(p, \gamma_0, \cdots, \gamma_{n-1})\in P\times G^{(n)}|\pi(p)=t(\gamma_0)\}$$ and define $$s_n: P\times_tG^{(n)}\rightarrow P ~\text{by}~ s_n(p, \gamma_1, \cdots, \gamma_n)=p\gamma_1\cdots\gamma_n.$$ Thus, we have the pullback vector bundle over $P\times_tG^{(n)}$ denoted by $s_n^*E$.

If we define $\Omega_c^{m, n}(P\times_tG; E)$ be the quotient of $\Omega_c^m(P\times_tG^{(n)}; s_n^*E)$
by the sections which are supported on
$\{(p, \gamma_0, \cdots, \gamma_{n-1}): \gamma_j~\mathrm{is~a~unit~for~some}~j\geq0\}$,
it follows from the discussion preceding \cite[(6.23)]{gorokhovsky2003local} that 
\[
\Omega_c^{m, n}(G)\otimes_{C_c^{\infty}(G)}\Gamma_c^{\infty}(P; E)\cong\Omega_c^{m, n}(P\times_tG; E).\]

We also have an isomorphism between $\Omega_c^{m', n'}(G)\otimes_{C_c^{\infty}(G)}\Omega_c^{m, n}(P\times_tG; E)$ and $\Omega_c^{m'+m, n'+n}(P\times_tG; E)$
by extending
the vector representation to a left action of 
$\Omega_c^\bullet(G)$ on
$\Omega_c^\bullet(P\times_t G; E)$ similar to Definition \ref{2024.05.18}.
\begin{definition}\label{2024.2.2}
Let $\omega\in\Omega_c^{*, k}(G)$ and $
F\in\Omega_c^{*, l}(P\times_t G; E)$. Then we define $\nu(\omega)F\in\Omega_c^{*, k+l}(P\times_t G; E)$ by
\begin{equation}
\begin{split}
(\nu(\omega)F)(p, &\gamma_0, \cdots, \gamma_{k+l-1})\\
:=
&\mathlarger{\sum}_{i=0}^{k-1}(-1)^i\sum_{\gamma\gamma'=\gamma_i}
\omega(\gamma_0; \cdots, \gamma, \gamma', \cdots, \gamma_{k-1})\wedge
F(p\gamma_0\cdots\gamma_{k-1}, \gamma_k, \cdots, \gamma_{k+l-1})\\
&+(-1)^k\sum_{\gamma\gamma'=\gamma_k}
\omega(\gamma_0; \cdots, \gamma_{k-1}, \gamma)\wedge
F(p\gamma_0\cdots\gamma, \gamma',\cdots, \gamma_{k+l-1})\\
&+\mathlarger{\sum}_{j=k+1}^{k+l-1}(-1)^j\sum_{\gamma\gamma'=\gamma_j}
\omega(\gamma_0; \cdots, \gamma_{k-1}, \gamma_k)\wedge
F(p\gamma_0\cdots\gamma_k, \cdots, \gamma, \gamma',\cdots, \gamma_{k+l-1})\\
&+(-1)^{k+l}\sum_{\gamma\gamma'=\gamma_{k+l}}
\omega(\gamma_0; \cdots, \gamma_{k-1}, \gamma_k)\wedge
F(p\gamma_0\cdots\gamma_k, \cdots, \gamma_{k+l-1}, \gamma)\gamma^{-1}.\\
\end{split}
\end{equation}
\end{definition}
A $C_c^{\infty}(G)$-valued inner product on $\Gamma_c^{\infty}(E_{\mathfrak{b}})\cong\Gamma_c^{\infty}(P; E)$ induced from \eqref{2022.2.28} is given by the formula
\[\langle u_1, u_2\rangle_{E_{\mathfrak{b}}\times_t G}(\gamma):=
\int_{Z_{t(\gamma)}}\langle u_1(p), u_2(p\gamma)\gamma^{-1}\rangle_E~\mu_{t(\gamma)}(p).\]
There is a pre-Hilbert $C_c^{\infty}(G)$-module structure on $\Gamma_c^{\infty}(P; E)$ (a parallel result of \cite[Lemma 3,23]{so2017non}
).
\begin{lemma}\label{2024.3.31}
For any $f\in C_c^{\infty}(G), u_1, u_2\in \Gamma_c^{\infty}(P; E)$, we have
\[\begin{split}(\langle u_1, u_2\rangle_{E_{\mathfrak{b}}\times_t G})^*=&\langle u_2, u_1\rangle_{E_{\mathfrak{b}}\times _tG},\\
f\star\langle u_1, u_2\rangle_{E_{\mathfrak{b}}\times_t G}=&
\langle\nu(f)u_1, u_2\rangle_{E_{\mathfrak{b}}\times_t G}.
\end{split}\]
\end{lemma}
\begin{proof}
For any $\gamma_0\in G$ and $\gamma_0\notin G^{(0)}$, we have for the first equality,
\[
\begin{split}
(\langle u_1, u_2\rangle_{E_{\mathfrak{b}}\times_t G})^*(\gamma_0)=&\overline{\langle u_1, u_2\rangle_{E_{\mathfrak{b}}\times_t G}(\gamma_0^{-1})}\\
=&
\int_{Z_{s(\gamma_0)}}\langle u_2(p'\gamma_0^{-1})\gamma_0, u_1(p')\rangle_E~\mu_{s(\gamma_0)}(p')\\
=&
\int_{Z_{t(\gamma_0)}}\langle u_2(p), u_1(p\gamma_0)\gamma_0^{-1}\rangle_E~\mu_{t(\gamma_0)}(p)\\
=&\langle u_2, u_1\rangle_{E_{\mathfrak{b}}\times_t G}(\gamma_0),
\end{split}
\]
because the metric on $E$ is $G$-invariant.

For the second equality,
\[
\begin{split}
(f\star\langle u_1, u_2\rangle_{E_{\mathfrak{b}}\times_t G})(\gamma_0)=&\sum_{\gamma\gamma'=\gamma_0}f(\gamma)\langle u_1, u_2\rangle_{E_{\mathfrak{b}}\times_t G}(\gamma')\\
=&\sum_{\gamma\gamma'=\gamma_0}f(\gamma)\int_{Z_{t(\gamma')}}\langle u_1(p'), u_2(p'\gamma')\gamma'^{-1}\rangle_E~\mu_{t(\gamma')}(p')\\
=&\int_{Z_{t(\gamma_0)}}\langle\sum_{\gamma\gamma'=\gamma_0}f(\gamma)u_1(p\gamma)\gamma^{-1}, u_2(p\gamma_0)\gamma_0^{-1}\rangle_E~\mu_{t(\gamma_0)}(p)\\
=&\langle\nu(f)u_1, u_2\rangle_{E_{\mathfrak{b}}\times_t G}(\gamma_0).
\end{split}
\]
The lemma is proved.
\end{proof}

\subsection{$\Omega_c^{\bullet}(G)$-linear operator}
According to \cite{LeichtnamPiazza+2005+169+233}, the local version of Connes' index theorem makes use of the heat kernel associated to a superconnection. 
In this subsection, we consider $\Omega_c^{\bullet}(G)$-linear smoothing operator, where the heat kernel is an example.
\begin{definition}\label{2024.6.14}
A $\bC$-linear operator $K: \Omega_c^{m, 0}(P\times_tG; E)\rightarrow \Omega_c^{m, k}(P\times_tG; E)$ is called $\Omega_c^{\bullet}(G)$-linear if for any $f\in C_c^{\infty}(G)$ and $F\in\Omega_c^{m, 0}(P\times_tG; E)$, 
$$K(\nu(f)F)=\nu(f)KF.$$
\end{definition}
We define a linear operator $K: \Omega_c^{m, *}(P\times_tG)\rightarrow \Omega_c^{m, *+k}(P\times_tG)$ by generalizing an $\Omega_c^{\bullet}(G)$-linear operator. Explicitly, $$K(\nu(\omega)F)=(-1)^{k(n+k')}\nu(\omega)KF,$$ where $\omega\in\Omega_c^{n, k'}(G)$.

Next, we consider a special type of $\Omega_c^{\bullet}(G)$-linear
operator, which is a smoothing operator along the fiber. 
We consider the manifold 
$$P\times_sG^{(k+l)}\times_tP:=
\{(p, \gamma_{k+l}, \cdots, \gamma_1, q)\left| q\in Z_{t(\gamma_1)},\right.
p\in Z_{s(\gamma_{k+l})}, \text{and}~~~\gamma_1, \cdots, \gamma_{k+l}\notin G^{(0)}\}$$
with maps $t(p, \gamma_{k+l}, \cdots, \gamma_1, q)=q$ and $s(p, \gamma_{k+l}, \cdots, \gamma_1, q)=p$.
Then we have a $G$-equivariant vector bundle over it denoted by 
$$\widehat E:=s^*E\otimes t^*E'.$$
Here $E$ and its dual $E'$ are $G$-equivariant vector bundles over $P$. 
\begin{definition}\label{2024.06.14}
An
$\Omega_c^{\bullet}(G)$-linear smoothing operator
$$K: \Omega_c^{\bullet}(P\times_tG; E)\rightarrow\Omega_c^{0, k}(G)\otimes_{C_c^{\infty}(G)}\Omega_c^{\bullet}(P\times_tG; E)
$$
is an operator of the form 
\begin{equation}\label{2024.5.15}
\begin{split}
&(KF)(p, \gamma_1, \cdots, \gamma_{k+l})\\
:=&(-1)^{kl}\int_{Z_{s(\gamma_{l})}}
K(p\gamma_1\cdots\gamma_{k+l}, \gamma_{k+l}, \cdots, \gamma_{l+1}, q)
F(q\gamma_{l}^{-1}\cdots\gamma_1^{-1}, \gamma_1, \cdots, \gamma_{l})
~\mu_{s(\gamma_{l})}(q),\\
\end{split}
\end{equation}
where the kernel 
$K(p\gamma_1\cdots\gamma_{k+l}, \gamma_k, \cdots, \gamma_1, q)$ is an element in $\Gamma_c^\infty(P\times_s G^{(k)}\times_t P; \widehat E)$
and satisfies Definition \ref{2024.6.14}.
\end{definition}
\begin{notation}
We denote by 
\[\Psi_{C_c^\infty(G)}^{-\infty}(\Omega_c^{\bullet}(P\times_tG; E), \Omega_c^{\bullet}(G)\otimes_{C_c^\infty(G)}\Omega_c^{\bullet}(P\times_tG; E))
\]
the set of left $\Omega_c^{\bullet}(G)$-linear smoothing operator $K$ acting on $\Omega_c^{\bullet}(P\times_tG; E)$ with a smooth integral kernel (cf. \cite[(6.24)]{gorokhovsky2003local}).
\end{notation}
In this paper, for simplicity, we always assume that 
$$K\in \Psi_{C_c^\infty(G)}^{-\infty}(\Omega_c^{\bullet}(P\times_tG; E), \Omega_c^{0, *}(G)\otimes_{C_c^\infty(G)}\Omega_c^{\bullet}(P\times_tG; E)).$$

By definition, each $\gamma\in G^{\pi(p)}$ induces isomorphisms
\[\begin{split}
A_\gamma&: Hom(E_q, E_p)\rightarrow
Hom(E_q, E_{p\gamma}),\\
B_\gamma&: Hom(E_q, E_p)\rightarrow
Hom(E_{q\gamma}, E_{p}).\\
\end{split}\]
In fact, for any $K\in \Psi_{C_c^\infty(G)}^{-\infty}(\Omega_c^{\bullet}(P\times_tG; E), \Omega_c^{0, k}(G)\otimes\Omega_c^{\bullet}(P\times_tG; E))$ and $F\in\Omega_c^{*, 0}(P\times_t G; E)$,
\[\begin{split}
(A_{\gamma}(K(p, \gamma_{k}, \cdots, \gamma_{1}, q)))
F(q):=&\big(K(p, \gamma_{k}, \cdots, \gamma_{1}, q)
F(q)\big)\gamma,\\
(B_\gamma (K(p, \gamma_{k}, \cdots, \gamma_{1}, q)))F(q\gamma)
:=&K(p, \gamma_{k}, \cdots, \gamma_{1}, q)(F(q\gamma)\gamma^{-1}).
\end{split}\]

The parallel properties of Example 3.11 in \cite{so2017non} and its preceding arguments are stated in the following proposition:
\begin{proposition}\label{2024.6.12}
Let $K\in \Psi_{C_c^\infty(G)}^{-\infty}(\Omega_c^{\bullet}(P\times_tG; E), \Omega_c^{0, 1}(G)\otimes\Omega_c^{\bullet}(P\times_tG; E))$,
we have
\begin{equation}\label{2024.6.17}
B_\gamma K(p, \gamma_{1}, q)=
K(p, \gamma^{-1}\gamma_1, q\gamma),
\quad\forall~\gamma\in G^{t(\gamma_1)},
\end{equation}
and
\begin{equation}\label{2022.7.1}
\sum_{\gamma\gamma'=\gamma_{2}}
A_{\gamma^{-1}}K(p\gamma, \gamma, q)
=0.
\end{equation}
\end{proposition}
\begin{proof}
For $F\in\Omega_c^{*, 0}(P\times_t G; E)=\Omega_c^*(P; E)$ (the same results hold for general $F\in\Omega_c^{*, l}(P\times_t G; E)$), 
we have $K(\nu(f)F)=\nu(f)KF$. 
Then we compute both sides of the preceding equation:
\begin{equation}\label{2021.1}
\begin{split}
&K(\nu(f)F)(p, \gamma_1)\\
=&\int_{Z_{t(\gamma_1)}}
K(p\gamma_1, \gamma_{1}, q)
(\nu(f)F)(q)
~\mu_{t(\gamma_1)}(q)\\
=&\int_{Z_{t(\gamma_1)}}
K(p\gamma_1, \gamma_{1}, q)
\big(\sum_{\gamma\gamma'=\gamma_1}
f(\gamma)F(q\gamma)\gamma^{-1}\big)~\mu_{t(\gamma_1)}(q)\\
=&\sum_{\gamma\gamma'=\gamma_1}
f(\gamma)
\int_{Z_{t(\gamma_1)}}
B_\gamma K(p\gamma_1, \gamma_{1}, q)F(q\gamma)~\mu_{t(\gamma_1)}(q),\\
\end{split}
\end{equation}

and
\begin{equation}\label{2021.1.}
\begin{split}
&\nu(f)(KF)(p, \gamma_1)\\
=&\sum_{\gamma\gamma'=\gamma_1}f(\gamma)
KF(p\gamma, \gamma')
-\sum_{\gamma\gamma'=\gamma_2}f(\gamma_1)
KF(p\gamma_1, \gamma)\gamma^{-1}
\\
=&\sum_{\gamma\gamma'=\gamma_{1}}f(\gamma)
\int_{Z_{t(\gamma')}}
K(p\gamma_1, \gamma', q')
F(q')
~\mu_{t(\gamma')}(q')\\
&-\sum_{\gamma\gamma'=\gamma_{2}}f(\gamma_1)
\left(\int_{Z_{s(\gamma_{1})}}K(p\gamma_1\gamma, \gamma, q)
F(q)
~\mu_{s(\gamma_{1})}(q)\right)\gamma^{-1}\\
=&\sum_{\gamma\gamma'=\gamma_{1}}f(\gamma)
\int_{Z_{t(\gamma)}}
K(p\gamma_1, \gamma^{-1}\gamma_1, q\gamma)
F(q\gamma)~\mu_{t(\gamma)}(q)\\
&-\sum_{\gamma\gamma'=\gamma_{2}}f(\gamma_1)
\int_{Z_{s(\gamma_{1})}}A_{\gamma^{-1}}K(p\gamma_1\gamma, \gamma, q)
F(q)
~\mu_{s(\gamma_{1})}(q).\\
\end{split}
\end{equation}

We rename $p'=p\gamma_1$ (then replacing $p'$ by $p$) and recall that $G^x:=\{\gamma\in G|t(\gamma)=x\}$.

Comparing the first term in \eqref{2021.1} and \eqref{2021.1.}, we get
\[
B_\gamma K(p, \gamma_{1}, q)=
K(p, \gamma^{-1}\gamma_1, q\gamma),
\quad\forall~\gamma\in G^{t(\gamma_1)}.
\]
Comparing the second term in \eqref{2021.1} and \eqref{2021.1.}, we get
\[
\int_{Z_{s(\gamma_{1})}}\sum_{\gamma\gamma'=\gamma_{2}}
A_{\gamma^{-1}}K(p\gamma, \gamma, q)
F(q)
~\mu_{s(\gamma_{1})}(q)=0.
\]
This implies that
\[
\sum_{\gamma\gamma'=\gamma_{2}}
A_{\gamma^{-1}}K(p\gamma, \gamma, q)
=0.
\]
This completes the proof.
\end{proof}
\begin{remark}
In fact the arguments in above proposition hold for any $K\in \Psi_{C_c^\infty(G)}^{-\infty}(\Omega_c^{\bullet}(P\times_tG; E), \Omega_c^{0, k}(G)\otimes\Omega_c^{\bullet}(P\times_tG; E))$, and one can obtain similar necessary conditions.
\end{remark}
If we define multiplication of smooth kernels in $\Gamma_c^\infty(P\times_s G^{(k+l)}\times_t P; \widehat E)$ by
$$(K_1\star K_2)(p, \gamma_{k+l}, \cdots, \gamma_{1},  q):=
(-1)^{kl}
\int_{Z_{s(\gamma_l)}}
K_1(p, \gamma_{k+l}, \cdots, \gamma_{l+1}, p')
K_2(p', \gamma_l, \cdots, \gamma_{1}, q)
~\mu_{s(\gamma_l)}(p'),$$
then by direct calculation using Equation \eqref{2024.5.15} and kernel multiplication formula, 
we get
$$K_2(K_1F)=(K_2\star K_1)F.$$

\subsection{The Bismut superconnection}
Recall the grading $E^{\bullet}:=E\otimes\wedge^\bullet V_P'$.

By \cite[(6.13)]{gorokhovsky2003local},
we may fix a non-negative function $h\in C_c^{\infty}(P)$ such that
\begin{equation}\label{2022.10.4}
\sum_{\gamma\in G^{\pi(p)}}h(p\gamma)=1
\end{equation} 
for all $p\in P$, and we define
\[
\begin{split}
\nabla^{0, 1}: \Gamma_c^{\infty}(P; E^\bullet)&\longrightarrow \Omega_c^{0, 1}(P\times_t G; E^\bullet)\\
F&\longmapsto\nabla^{0, 1}F: P\times_t G\longrightarrow s_1^*E^\bullet
\end{split}
\]by
$$(\nabla^{0, 1}F)(p_0, \gamma_0)=
\left(F(p_0)\gamma_0\right)h(p_0\gamma_0).$$
Then we can define a noncommutative superconnection on the Bismut bundle.

\begin{definition}
Let $D:=D_B+\nabla^{0, 1}$ denote a superconnection over the Bismut bundle and $D':=D_B'+\nabla^{0, 1}$ denote its adjoint. 
\end{definition}
Put $D(u)=uD+(1-u)D', u\in [0, 1]$. 
Similar to the paragraph above Definition \ref{2024.7.2}, given $p\in\pi^{-1}(t(U))$, we define an action of $S(G)$ on $P$ by
$$\theta_{U}(p)=p\cdot t|_U^{-1}(\pi(p)).$$
Then we have an action on $\Gamma_c^{\infty}(P; E^\bullet)$, which is induced by $\theta_{U}$ and is given by
$$([\theta_{U}]^*F)(p)=F(p\cdot t|_U^{-1}(\pi(p)))\cdot (t|_U^{-1}(\pi(p)))^{-1}.$$
A (commutative) superconnection $\nabla$ on $E^\bullet$ is called $G$-invariant if 
\begin{equation}\label{024.7.2}
\nabla([\theta_{U}]^*F)=[\theta_{U}]^*(\nabla F).
\end{equation}
In the following, we assume that all the (commutative) superconnections are $G$-invariant
and we define the exterior differential operator according to Definition \ref{2022.9.20}.
Then we have a parallel result of \cite[Lemma 3.22]{so2017non}.
\begin{lemma}
The operator $D_B+\nabla^{0, 1}$ satisfies the axiom in Definition \ref{2024.6.15}, it is a connection of the $C_c^{\infty}(G)$-module $\Gamma_c^{\infty}(P; E^\bullet)$.
\end{lemma}
\begin{proof}For $f\in C_c^{\infty}(G)$ and $F\in \Gamma_c^{\infty}(P; E^\bullet)$, we first verify that
$$D_B(\nu(f)F)=\nu(f)(D_BF)+\nu(d_1f)F.$$
By (iv) of Proposition \ref{2024.4.8}, $f=\sum_i f_i$, where $supp(f_i)$ is in a bisection $U_i$. Thus we need to prove
$$D_B(\nu(f_i)F)=\nu(f_i)(D_BF)+\nu(d_1f_i)F.$$
From Equation \eqref{024.7.2}, the exterior differential operator $D_B$ on $E^\bullet$ commutes with $[\theta_{U_i}]^*$,
it follows that
\[\begin{split}
D_B\big(\nu(f_i)F\big)=&D_B\big(((t|_{U_i}^{-1}\circ\pi)^*f_i)([\theta_{U_i}]^*F)\big)\\
=&(t|_{U_i}^{-1}\circ\pi)^*f_i~D_B\big([\theta_{U_i}]^*F\big)+(t|_{U_i}^{-1}\circ\pi)^*(d_1f_i)~[\theta_{U_i}]^*F\\
=&(t|_{U_i}^{-1}\circ\pi)^*f_i~[\theta_{U_i}]^*(D_BF)+\nu(d_1f_i)F\\
=&\nu(f_i)(D_BF)+\nu(d_1f_i)F.
\end{split}\]

Then we verify that
$$\nabla^{0, 1}(\nu(f)F)(p\gamma_0^{-1}, \gamma_0)=
\nu(f)(\nabla^{0, 1}F)(p\gamma_0^{-1}, \gamma_0)+
\nu(d_2f)F(p\gamma_0^{-1}, \gamma_0).$$
We have
\[
\begin{split}
\nabla^{0, 1}(\nu(f)F)(p\gamma_0^{-1}, \gamma_0)&=
\big((\nu(f)F)(p\gamma_0^{-1})\gamma_0\big)h(p)
=\sum_{t(\gamma)=\pi(p\gamma_0^{-1})}
f(\gamma)
(F(p\gamma_0^{-1}\gamma)\gamma^{-1}\gamma_0)h(p),\\
\nu(f)(\nabla^{0, 1}F)(p\gamma_0^{-1}, \gamma_0)&=
\sum_{\gamma\gamma'=\gamma_0}f(\gamma)\nabla^{0, 1}F(p\gamma_0^{-1}\gamma, \gamma')-
f(\gamma_0)\sum_{t(\gamma)=s(\gamma_0)}\nabla^{0, 1}F(p, \gamma)\gamma^{-1}
\\&=
\sum_{\gamma\gamma'=\gamma_0}f(\gamma)(
F(p\gamma_0^{-1}\gamma)\gamma^{-1}\gamma_0)h(p)-
f(\gamma_0)\sum_{t(\gamma)=s(\gamma_0)}
(F(p)\gamma\gamma^{-1})h(p\gamma),\\
(\nu(d_2f)F)(p\gamma_0^{-1}, \gamma_0)
&=\sum_{\gamma\gamma'=\gamma_0}
\chi_{G^{(0)}}(\gamma)f(\gamma')F(p)
-\sum_{\gamma\gamma'=\gamma_1}
\chi_{G^{(0)}}(\gamma_0)f(\gamma)F(p\gamma)\gamma^{-1}
=f(\gamma_0)F(p)
\end{split}
\]
for $\gamma_0\notin G^{(0)}.$
\end{proof}
\begin{corollary}
The curvature $D(u)^2$ are $\Omega_c^{\bullet}(G)$-linear. 
\end{corollary}
\begin{remark}\label{2024.2.1}
For $F\in\Gamma_c^{\infty}(E)\bigotimes_{C_c^{\infty}(G)}\Omega_c^{*,l}(G),$ we let
\[\begin{split}
(\nabla^{0, 1}F)(p\gamma_{l}^{-1}\cdots\gamma_0^{-1}, \gamma_0, \cdots, \gamma_{l})
:=&(-1)^l\big(F(p\gamma_{l}^{-1}\cdots\gamma_0^{-1}, \gamma_0, \cdots, \gamma_{l-1})\gamma_l\big)h(p).\\
\end{split}\]
We also have $\nabla^{0, 1}(\nu(\omega)F)=
(-1)^l\nu(\omega)(\nabla^{0, 1}F)+\nu(d_2\omega)F$, for any $\omega\in\Omega_c^{0, l}(G)$ and $F\in\Gamma_c^{\infty}(P; E)$.
\end{remark}

The following lemma is analogous to the statement above \cite[Lemma 9.15]{berline2003heat}.
\begin{lemma}
If $K$ is an $\Omega_c^{\bullet}(G)$-linear smoothing operator, then $[D, K ]$ is also an $\Omega_c^{\bullet}(G)$-linear smoothing operator.
\end{lemma}
\begin{proof}
Since the first and third terms of $D_B$ are $\Omega_c^{\bullet}(G)$-linear operator, we only need to verify that $[L^{E_{\mathfrak{b}}^\bullet}+\nabla^{0,1}, K]$ is $\Omega_c^{\bullet}(G)$-linear. Consider the case when $k=1$. Other cases are similar.

First, it's easy to verify
\[\begin{split}
[L^{E_{\mathfrak{b}}^\bullet}, K](\nu(f)F)&=(L^{E_{\mathfrak{b}}^\bullet}K+KL^{E_{\mathfrak{b}}^\bullet})(\nu(f)F)\\
&=\nu(f)L^{E_{\mathfrak{b}}^\bullet}KF+\nu(d_1f)KF+\nu(f)KL^{E_{\mathfrak{b}}^\bullet}F-\nu(d_1f)KF\\
&=\nu(f)[L^{E_{\mathfrak{b}}^\bullet}, K]F.
\end{split}\]
Thus we are reduced to showing that $[\nabla^{0,1}, K]$ is $\Omega_c^{\bullet}(G)$-linear. Observe that
\[\begin{split}
[\nabla^{0,1}, K](\nu(f)F)&=(\nabla^{0,1}K+K\nabla^{0,1})(\nu(f)F)
\end{split}\]
and
\[\begin{split}
\nu(f)[\nabla^{0,1}, K]F=\nu(f)\nabla^{0,1}KF+\nu(d_2f)KF+\nu(f)K\nabla^{0,1}F-\nu(d_2f)KF.\\
\end{split}\]
Then the lemma follows from the equations 
\[\begin{split}\nabla^{0,1}K(\nu(f)F)=\nu(f)\nabla^{0,1}KF+\nu(d_2f)KF,\\
(K\nabla^{0,1})(\nu(f)F)=\nu(f)K\nabla^{0,1}F-\nu(d_2f)KF.
\end{split}\]
By Remark \ref{2024.2.1}, we compute directly
\[\begin{split}
&\nabla^{0,1}(\nu(f)KF)(p\gamma_1^{-1}\gamma_0^{-1}, \gamma_0, \gamma_1)\\
=&-(\nu(f)KF)(p\gamma_1^{-1}\gamma_0^{-1}, \gamma_0)\gamma_1h(p)\\
=&-\sum_{\gamma\gamma'=\gamma_0}f(\gamma)
KF(p\gamma_{1}^{-1}\gamma_0^{-1}\gamma, \gamma')\gamma_1h(p)
+\sum_{\gamma\gamma'=\gamma_1}f(\gamma_0)
(KF(p\gamma_{1}^{-1}, \gamma)\gamma^{-1})\gamma_1h(p),
\end{split}\]

\[\begin{split}
&\nu(f)(\nabla^{0,1}KF)(p\gamma_1^{-1}\gamma_0^{-1}, \gamma_0, \gamma_1)\\
=&-\sum_{\gamma\gamma'=\gamma_0}f(\gamma)
\big(KF(p\gamma_{1}^{-1}\gamma_0^{-1}\gamma, \gamma')\gamma_1\big)h(p)\\
&+\sum_{\gamma\gamma'=\gamma_1}f(\gamma_0)
\big(KF(p\gamma_{1}^{-1}, \gamma)\gamma'\big)h(p)
-\sum_{\gamma\gamma'=\gamma_2}f(\gamma_0)
\big(KF(p'\gamma^{-1}\gamma_{1}^{-1}, \gamma_1)\gamma\big)\gamma^{-1}h(p\gamma)\\
=&\nabla^{0,1}(\nu(f)KF)(p\gamma_1^{-1}\gamma_0^{-1}, \gamma_0, \gamma_1)
-f(\gamma_0)KF(p\gamma_{1}^{-1}, \gamma_1),
\end{split}\]
and from Definition \ref{2024.2.2},
\[\begin{split}
&(\nu(d_2f)KF)(p\gamma_1^{-1}\gamma_0^{-1}, \gamma_0, \gamma_1)\\
=&\bigg\{
\sum_{\gamma\gamma'=\gamma_0}d_2f(\gamma; \gamma')\wedge(KF)(p\gamma_1^{-1}, \gamma_1)\\
&-\sum_{\gamma\gamma'=\gamma_1}d_2f(\gamma_0; \gamma)\wedge(KF)(p\gamma_1^{-1}\gamma, \gamma')
+\sum_{t(\gamma)=s(\gamma_1)}d_2f(\gamma_0; \gamma_1)\wedge(KF)(p, \gamma)\gamma^{-1}
\bigg\}\\
=&f(\gamma_0)KF(p\gamma_{1}^{-1}, \gamma_1).
\\
\end{split}\]
Then we prove that $K\nabla^{0,1}(\nu(f)F)=\nu(f)K\nabla^{0,1}F-\nu(d_2f)KF$.
Recall Equation \eqref{2024.5.15}, we get
\[\begin{split}
&K\nabla^{0,1}(\nu(f)F)(p\gamma_1^{-1}\gamma_0^{-1},\gamma_0, \gamma_1)\\
=&-\int_{Z_{t(\gamma_1)}}K(p, \gamma_1, q)
\sum\limits_{t(\gamma)=\pi(q\gamma_0^{-1})}f(\gamma)(F(q\gamma_0^{-1}\gamma)\gamma^{-1}\gamma_0)h(q)~\mu_{t(\gamma_1)}(q),\\
&\nu(d_2f)KF(p\gamma_1^{-1}\gamma_0^{-1}, \gamma_0, \gamma_1)\\
=&f(\gamma_0)\int_{Z_{t(\gamma_1)}}K(p, \gamma_1, q)F(q)~\mu_{t(\gamma_1)}(q)
\end{split}\]
and
\[\begin{split}
&\nu(f)K\nabla^{0,1}F(p\gamma_1^{-1}\gamma_0^{-1}, \gamma_0, \gamma_1)\\
=&-\sum_{\gamma\gamma'=\gamma_0}f(\gamma)
\int_{Z_{t(\gamma_1)}}K(p, \gamma_1, q)
(F(q\gamma'^{-1})\gamma')h(q)~\mu_{t(\gamma_1)}(q)\\
&+\sum_{\gamma\gamma'=\gamma_1}f(\gamma_0)
\int_{Z_{s(\gamma)}}K(p, \gamma', q)
(F(q\gamma^{-1})\gamma)h(q)~\mu_{s(\gamma)}(q)\\
&-\sum_{\gamma\gamma'=\gamma_2}f(\gamma_0)
\left(\int_{Z_{s(\gamma_1)}}K(p\gamma, \gamma, q)
(F(q\gamma_1^{-1})\gamma_1)h(q)~\mu_{s(\gamma_1)}(q)\right)\gamma^{-1}\\
=&K\nabla^{0,1}(\nu(f)F)(p\gamma_1^{-1}\gamma_0^{-1}, \gamma_0, \gamma_1)
+\sum_{\gamma\gamma'=\gamma_1}f(\gamma_0)
\int_{Z_{t(\gamma_1)}}K(p, \gamma', q'\gamma)
(F(q')\gamma)h(q'\gamma)~\mu_{t(\gamma_1)}(q')\\
&-f(\gamma_0)
\int_{Z_{s(\gamma_1)}}
\sum_{\gamma\gamma'=\gamma_2}A_{\gamma^{-1}}K(p\gamma, \gamma, q)
(F(q\gamma_1^{-1})\gamma_1)h(q)~\mu_{s(\gamma_1)}(q).\\
\end{split}\]
The claim follows from Equations \eqref{2024.6.17}, \eqref{2022.10.4} and \eqref{2022.7.1}.
\end{proof}

\subsection{Trace}
In this subsection, we generalize the pointwise trace mentioned in Subsection \ref{Chern class} to the Bismut bundle. 
Our definition follows \cite[Section 3]{gorokhovsky2003local}.

We begin with defining the trace for $\Omega_c^{\bullet}(G)$-linear smoothing operators
(cf. \cite[(6.28)]{gorokhovsky2003local}).
\begin{definition}
The trace of the $\Omega_c^{\bullet}(G)$-linear operator is defined by
$$
Tr_{<e>}:
\Psi_{C_c^\infty(G)}^{-\infty}(\Omega_c^{\bullet}(P\times_tG; E), \Omega_c^{\bullet}(G)\otimes_{C_c^\infty(G)}\Omega_c^{\bullet}(P\times_tG; E))
\longrightarrow\Omega_c^{\bullet}(G)_{\rm{Ab}}
$$
\begin{equation}\begin{split}\label{2022.5.15}
Tr_{<e>}(K)(\gamma_0; \cdots, \gamma_n):=&
\chi_{G^{(0)}}(\gamma_0\cdots\gamma_n)\int_{Z_{t(\gamma_0)}}h(p)\\
\bigg[\sum_{\gamma\gamma'=\gamma_n}
tr&\big(A_{\gamma'}K(p(\gamma')^{-1}, \gamma, \gamma_{n-1}, \cdots, \gamma_1, p)\big)\chi_{G^{(0)}}(\gamma_0)\\
+\mathlarger{\sum}_{i=1}^{n-1}(-1)^i\sum_{\gamma\gamma'=\gamma_{n-i}}
tr&\big(A_{\gamma_n}K(p\gamma_n^{-1}, \gamma_{n-1}, \cdots, \gamma', \gamma, \cdots, \gamma_1, p)\big)\chi_{G^{(0)}}(\gamma_0)\\
+(-1)^{n}tr&\big(A_{\gamma_n}K(p\gamma_n^{-1},  \gamma_{n-1}, \cdots, \gamma_0, p)\big)\bigg]~\mu_{t(\gamma_0)}(p).
\end{split}\end{equation}
\end{definition}
Note that
\[tr\big(A_{\gamma_n}K(p\gamma_n^{-1},  \gamma_{n-1}, \cdots, \gamma_0, p)\big)\in\wedge^m T^*_{t(\gamma_0)}G^{(0)},\]
because
\[A_{\gamma_n}K(p\gamma_n^{-1}, \gamma_{n-1}, \cdots, \gamma_0, p)\in E_p\otimes E'_p\otimes\wedge^m T^*_{t(\gamma_0)}G^{(0)}.
\]
\begin{proposition}[cf. {\cite[Proposition 3]{gorokhovsky2003local}}]
$
Tr_{<e>}: \Psi_{C_c^\infty(G)}^{-\infty}(\Omega_c^{\bullet}(P\times_tG; E), \Omega_c^{\bullet}(G)\otimes_{C_c^\infty(G)}\Omega_c^{\bullet}(P\times_tG; E))
\longrightarrow\Omega_c^{\bullet}(G)_{\rm{Ab}}
$ is a trace.
\end{proposition} 
 The following theorem generalizes \cite[Lemma 9.15]{berline2003heat} to the \'etale case.
\begin{theorem}\label{2024.03.04}
Let $K\in \Psi_{C_c^\infty(G)}^{-\infty}(\Omega_c^{\bullet}(P\times_tG; E), \Omega_c^{*, k}(G)\otimes\Omega_c^{\bullet}(P\times_tG; E))$. 
Then
$$(d_1+d_2)(Tr_{<e>}(K))=Tr_{<e>}([D_B+\nabla^{0,1}, K]),$$
where we recall that $D_B$ is defined as in Definition \ref{2022.9.20},
$\nabla^{0,1}$ is defined as in 
Remark \ref{2024.2.1}.
\end{theorem}

\subsection{Chern character and main theorem}
We extend $Tr_{<e>}$ to the supertrace $sTr_{<e>}$ by using the induced $\bZ_2$-grading on $E^{\bullet}$ (cf. \cite[Subsection 1.3]{berline2003heat}).
Since we do not consider the $t\rightarrow 0$ and $t\rightarrow\infty$ behaviors of the heat kernel in this paper, we restrict our attention to the case $t=1$, the other cases are similar.

The heat kernel of $D(u)^2$, denoted by $e^{-D(u)^2}$, is constructed using a generalization of the Volterra series (cf. \cite[Appendix in Chapter 9]{berline2003heat}, \cite[(6.42)]{gorokhovsky2003local}).
Following \cite[Definition 9.16]{berline2003heat}, we have
\begin{definition}\label{2024.3.4}
The noncommutative Chern form is defined by
\[\begin{split}\operatorname{Ch}(D(u)):=&sTr_{<e>}(e^{-D(u)^2})\in \Omega_c^{\bullet}(G)_{\rm{Ab}}.\\
\end{split}
\]
\end{definition}
We thus arrive at the main result of this paper.
\begin{theorem}
The noncommutative differential form $\operatorname{Ch}(D(u))$ defined in Definition \ref{2024.3.4} is closed.
\end{theorem}
\begin{proof}
Using Theorem \ref{2024.03.04}, we have
\[\begin{split}
(d_1+d_2)(sTr_{<e>}(e^{-D(u)^2}))
&=sTr_{<e>}([D_B+\nabla^{0,1}, e^{-D(u)^2}])\\
&=sTr_{<e>}([D_B+\nabla^{0,1}-D(u), e^{-D(u)^2}])+sTr_{<e>}([D(u), e^{-D(u)^2}])\\
&=sTr_{<e>}([(1-u)(D_B-D_B'), e^{-D(u)^2}])\\
&=0.
\end{split}\]
The last line is obtained since 
the difference of the two superconnections is an $\Omega_c^{\bullet}(G)$-linear operator.
This proves that $\operatorname{Ch}(D(u))$ is closed.
\end{proof}
As an even form, $\operatorname{Ch}(D(u))$ defines a
cohomology class in $H^{even}(\Omega_c^{\bullet}(G)_{\rm{Ab}})$.

\section{Proof of Theorem \ref{2024.03.04}}\label{Proof of Theorem}

The goal of this section is to prove one of the main results of the paper, Theorem \ref{2024.03.04}, whose proof proceeds along the lines of \cite[Lemma 14]{LeichtnamPiazza+2005+169+233}
and \cite[Lemma 3.19]{so2017non}.

We recall that Theorem \ref{2024.03.04} states 
$$(d_1+d_2)(Tr_{<e>}(K))=Tr_{<e>}([D_B+\nabla^{0,1}, K]).$$
We need to prove 
\begin{equation}d_1(Tr_{<e>}(K))=Tr_{<e>}([D_B, K])
\end{equation}
and 
\begin{equation}\label{2024.5.20}
d_2(Tr_{<e>}(K))=Tr_{<e>}([\nabla^{0,1}, K]).
\end{equation}
In both cases, we assume $k=1$, the other cases are proved similarly.
\begin{lemma}
We have
$$d_1(Tr_{<e>}(K))=Tr_{<e>}([D_B, K]).$$
\end{lemma}
\begin{proof}
By \cite[Lemma 9.15]{berline2003heat},
\[\begin{split}
d_1(Tr_{<e>}(K))(\gamma_0; \gamma_1)=&
-\chi_{G^{(0)}}(\gamma_0\gamma_1)\int_{Z_{t(\gamma_0)}}d_1h(p)tr\big(A_{\gamma_0^{-1}}K(p\gamma_0, \gamma_0, p)\big)
~\mu_{t(\gamma_0)}(p)\\
&-\chi_{G^{(0)}}(\gamma_0\gamma_1)\int_{Z_{t(\gamma_0)}}h(p)tr\big(A_{\gamma_0^{-1}}[D_B, K](p\gamma_0, \gamma_0, p)\big)
~\mu_{t(\gamma_0)}(p)\\
\end{split}\]
and the first term above is equal to zero. In fact, by \eqref{2022.10.4} and $\sum_{\gamma\gamma'=\gamma_0}=\sum_{\gamma\in G^{t(\gamma_0)}}$, we have
\[\begin{split}
\Big(&d_1(Tr_{<e>}(K))-Tr_{<e>}([D_B, K])\Big)(\gamma_0; \gamma_1)\\
=&
-\chi_{G^{(0)}}(\gamma_0\gamma_1)\int_{Z_{t(\gamma_0)}}d_1h(p)tr\big(A_{\gamma_0^{-1}}K(p\gamma_0, \gamma_0, p)\big)
~\mu_{t(\gamma_0)}(p)\\
=&-\chi_{G^{(0)}}(\gamma_0\gamma_1)\int_{Z_{t(\gamma_0)}}\sum_{\gamma\gamma'=\gamma_0}h(p\gamma)d_1h(p)tr\big(A_{\gamma_0^{-1}}K(p\gamma_0, \gamma_0, p)\big)
~\mu_{t(\gamma_0)}(p)\\
=&-\chi_{G^{(0)}}(\gamma_0\gamma_1)
\sum_{\gamma\gamma'=\gamma_0}\int_{Z_{s(\gamma)}}h(q)d_1h(q\gamma^{-1})
tr\big(A_{\gamma_1}K(q\gamma^{-1}\gamma_1^{-1}, \gamma_1^{-1}, q\gamma^{-1})\big)
~\mu_{s(\gamma)}(q).\\
\end{split}\]
If we let $$\omega(a; b)=\int_{Z_{t(a)}}h(q)d_1h(qa)
tr\big(A_{b^{-1}}K(qab, b, qa)\big)
~\mu_{t(a)}(q),$$ then we observe that $\omega\in\Omega_c^{*, 1}(G)$ and can wirte $\omega=\sum\limits_{i, j}\omega_{i, j}$ with $supp(\omega_{i, j})\subseteq U_i\hat\times V_j$ by Proposition \ref{2024.04.08}. Hence, according to Corollary \ref{2024.03.31}, there exist 
$1_i, 1_j\in C_c^{\infty}(G)$ with $supp(1_i)\subseteq U_i^{-1}$ and $supp(1_j)\subseteq V_j^{-1}$.

We claim that 
\[\begin{split}
&\sum_{i, j}\big[1_{i}, \omega_{i, j}\star 1_{j}\big](\gamma_0; \gamma_1)\\
=&-\chi_{G^{(0)}}(\gamma_0\gamma_1)
\sum_{\gamma\gamma'=\gamma_0}\int_{Z_{s(\gamma)}}h(q)d_1h(q\gamma^{-1})
tr\big(A_{\gamma_1}K(q\gamma^{-1}\gamma_1^{-1}, \gamma_1^{-1}, q\gamma^{-1})\big)
~\mu_{s(\gamma)}(q)\\
=&-\chi_{G^{(0)}}(\gamma_0\gamma_1)\sum_{\gamma\gamma'=\gamma_0}
\omega(\gamma^{-1}; \gamma_0).
\end{split}\]
We have
\begin{equation}\label{2024.1.13}
\sum_{i, j}\big[1_{i}, \omega_{i, j}\star 1_{j}\big](\gamma_0; \gamma_1)
=\sum_{i, j}\big(1_{i}\star \omega_{i, j}\star 1_{j}-\omega_{i, j}\star 1_{j}\star1_{i}\big)(\gamma_0; \gamma_1).
\end{equation}
The first term of the above equation is equal to 
\begin{equation}\begin{split}
&\sum_{i, j}(1_{i}\star \omega_{i, j}\star 1_{j})(\gamma_0; \gamma_1)\\
=&\sum_{i, j}\sum_{\gamma\gamma'=\gamma_0}
1_{i}(\gamma)\bigg(\sum_{ab=\gamma_1}\omega_{i, j}(\gamma'; a)(1_{j})(b)-\sum_{ab=\gamma'}\omega_{i, j}(a; b)(1_{j})(\gamma_1)\bigg)\\
=&-\sum_{i, j}\sum_{\gamma\gamma'=\gamma_0}
\omega_{i, j}(\gamma^{-1}; \gamma_0)1_{j}(\gamma_1)\\
=&-\sum_{i, j}\sum_{\gamma\gamma'=\gamma_0}
\omega_{i, j}(\gamma^{-1}; \gamma_0)\chi_{G^{(0)}}(\gamma_0\gamma_1)\\
=&-\chi_{G^{(0)}}(\gamma_0\gamma_1)\sum_{\gamma\gamma'=\gamma_0}
\omega(\gamma^{-1}; \gamma_0).
\end{split}\end{equation}
Similarly, the second term of \eqref{2024.1.13} is equal to
\begin{equation}\label{2023.02.24}
\begin{split}
&\sum_{i, j}(\omega_{i, j}\star 1_{j}\star1_{i})(\gamma_0; \gamma_1)\\
=&-\sum_{i, j}\sum_{\gamma\gamma'=\gamma_0}
(\omega_{i, j}\star 1_j)(\gamma; \gamma')1_i(\gamma_1)+\sum_{i, j}\sum_{\gamma\gamma'=\gamma_1}
(\omega_{i, j}\star 1_j)(\gamma_0; \gamma)1_i(\gamma')\\
=
&-\sum_{i, j}\sum_{\gamma\gamma'=\gamma_0}
\bigg(\sum_{ab=\gamma'}\omega_{i, j}(\gamma; a)1_j(b)1_i(\gamma_1)-\sum_{ab=\gamma}\omega_{i, j}(a; b)1_j(\gamma')1_i(\gamma_1)\bigg)\\
&+\sum_{i, j}\sum_{\gamma\gamma'=\gamma_1}
\bigg(\sum_{ab=\gamma}\omega_{i, j}(\gamma_0; a)1_j(b)1_i(\gamma')-\sum_{ab=\gamma_0}\omega_{i, j}(a; b)1_j(\gamma)1_i(\gamma')\bigg)\\
=&-
\bigg(\sum_{t(a)=s(\gamma_0)}\omega(\gamma_0; a)\chi_{G^{(0)}}(\gamma_0\gamma_1)-\sum_{t(\gamma'^{-1})=s(\gamma_0)}\omega(\gamma_0; \gamma'^{-1})\chi_{G^{(0)}}(\gamma_0\gamma_1)\bigg)\\
&+\sum_{t(\gamma)=s(\gamma_0)}
\bigg(\omega(\gamma_0; \gamma)\chi_{G^{(0)}}(\gamma_0\gamma_1)-\omega(\gamma_0\gamma; \gamma^{-1})\chi_{G^{(0)}}(\gamma_0\gamma_1)\bigg).
\end{split}
\end{equation}
By using \eqref{2022.7.1}, for all $\gamma_0$, $$
\sum_{t(a)=s(\gamma_0)}\omega(\gamma_0; a)=
\sum_{t(\gamma'^{-1})=s(\gamma_0)}\omega(\gamma_0; \gamma'^{-1})=
\sum_{t(\gamma)=s(\gamma_0)}
\omega(\gamma_0; \gamma)=0.$$
According to \eqref{2024.6.17}, the integrand $tr(A_\gamma K(q\gamma_0, \gamma^{-1}, q\gamma_0\gamma))$ in $\omega(\gamma_0\gamma; \gamma^{-1})$ is independent of $\gamma$. Because $\sum\limits_{\gamma\in G^{\pi(p)}}d_1h(p\gamma)=0
$, we finally get that \eqref{2023.02.24} vanishes.
The claim is proved.
\end{proof}
Next, we prove the second identity \eqref{2024.5.20}.
\begin{lemma}We have
$$d_2(Tr_{<e>}(K))=Tr_{<e>}([\nabla^{0,1}, K]).$$
\end{lemma}
\begin{proof}
Observe that
\[\begin{split}
\nabla^{0,1}KF(p\gamma_1^{-1}\gamma_0^{-1}, \gamma_0, \gamma_1)
=&-KF(p\gamma_1^{-1}\gamma_0^{-1}, \gamma_0)\gamma_1h(p)\\
=&-\int_{Z_{t(\gamma_0)}}h(p)
(A_{\gamma_1}(K(p\gamma_1^{-1}, \gamma_0, q)))F(q)~\mu_{t(\gamma_0)}(q),\\
K\nabla^{0,1}F(p\gamma_1^{-1}\gamma_0^{-1}, \gamma_0, \gamma_1)
=&-\int_{Z_{s(\gamma_0)}}K(p, \gamma_1, q)\nabla^{0,1}F(q\gamma_0^{-1}, \gamma_0)~\mu_{s(\gamma_0)}(q)\\
=&-\int_{Z_{t(\gamma_0)}}K(p, \gamma_1, q\gamma_0)(F(q)\gamma_0)h(q\gamma_0)~\mu_{t(\gamma_0)}(q)\\
=&-\int_{Z_{t(\gamma_0)}}(B_{\gamma_0^{-1}}(K(p, \gamma_1, q\gamma_0)))h(q\gamma_0)F(q)~\mu_{t(\gamma_0)}(q).\end{split}\]
Note that the second last line of the above formula is obtained by renaming $q'\gamma_0=q$ (then replacing $q'$ by $q$).
So we have
$$[\nabla^{0,1}, K](p, \gamma_1, \gamma_0, q)
=-\Big\{h(p)(A_{\gamma_1}(K(p\gamma_1^{-1}, \gamma_0, q)))+(B_{\gamma_0^{-1}}(K(p, \gamma_1, q\gamma_0)))h(q\gamma_0)\Big\}.
$$
Thus
\begin{equation}\label{2022.11.3}
\begin{split}
&Tr_{<e>}([\nabla^{0,1}, K])(\gamma_0;  \gamma_1, \gamma_2)\\
=&-\chi_{G^{(0)}}(\gamma_0\gamma_1\gamma_2)\int_{Z_{t(\gamma_0)}}h(p)
\sum_{\gamma\gamma'=\gamma_2}
h(p\gamma'^{-1})tr\big(A_{\gamma_2}K(p\gamma_2^{-1}, \gamma_1, p)\big)\chi_{G^{(0)}}(\gamma_0)
~\mu_{t(\gamma_0)}(p)\\
&-\chi_{G^{(0)}}(\gamma_0\gamma_1\gamma_2)\int_{Z_{t(\gamma_0)}}h(p)
\sum_{\gamma\gamma'=\gamma_2}
tr\big(A_{\gamma'}B_{\gamma_1^{-1}}K(p\gamma'^{-1}, \gamma, p\gamma_1)\big)h(p\gamma_1)
\chi_{G^{(0)}}(\gamma_0)~\mu_{t(\gamma_0)}(p)
\\
&+\chi_{G^{(0)}}(\gamma_0\gamma_1\gamma_2)\int_{Z_{t(\gamma_0)}}h(p)\sum_{\gamma\gamma'=\gamma_1}
h(p\gamma_2^{-1})tr\big(A_{\gamma^{-1}}K(p\gamma_2^{-1}\gamma'^{-1}, \gamma, p)\big)\chi_{G^{(0)}}(\gamma_0)
~\mu_{t(\gamma_0)}(p)\\
&+\chi_{G^{(0)}}(\gamma_0\gamma_1\gamma_2)\int_{Z_{t(\gamma_0)}}h(p)\sum_{\gamma\gamma'=\gamma_1}
tr\big(A_{\gamma_2}B_{\gamma^{-1}}K(p\gamma_2^{-1}, \gamma', p\gamma)\big)h(p\gamma)
\chi_{G^{(0)}}(\gamma_0)~\mu_{t(\gamma_0)}(p)
\\
&-\chi_{G^{(0)}}(\gamma_0\gamma_1\gamma_2)\int_{Z_{t(\gamma_0)}}h(p)h(p\gamma_2^{-1})tr\big(A_{\gamma_0^{-1}}K(p\gamma_2^{-1}\gamma_1^{-1},  \gamma_0, p)\big)~\mu_{t(\gamma_0)}(p)\\
&-\chi_{G^{(0)}}(\gamma_0\gamma_1\gamma_2)\int_{Z_{t(\gamma_0)}}h(p)tr\big(A_{\gamma_2}B_{\gamma_0^{-1}}K(p\gamma_2^{-1}, \gamma_1, p\gamma_0)\big)h(p\gamma_0)
~\mu_{t(\gamma_0)}(p).\\
\end{split}\end{equation}
We focus on the right-hand-side of \eqref{2022.11.3}. The first term is equal to 
\[\begin{split}
&\chi_{G^{(0)}}(\gamma_0)
\chi_{G^{(0)}}(\gamma_1\gamma_2)\int_{Z_{t(\gamma_1)}}h(p)\Big(-tr\big(A_{\gamma_2}K(p\gamma_2^{-1}, \gamma_1, p)\big)
\Big)~\mu_{t(\gamma_1)}(p)\\
=&
d_2(Tr_{<e>}(K))(\gamma_0; \gamma_1, \gamma_2),
\end{split}\]
since 
$\sum\limits_{\gamma\gamma'=\gamma_2}
h(p\gamma'^{-1})=1$.
The second and third terms are equal to zero by \eqref{2022.7.1}.
We define $\mu\in\Omega_c^{*, 1}(G)$ by
\begin{equation}\label{2024.06.07}
\mu(a; b)=\int_{Z_{t(a)}}h(p)h(pa)
tr\big(A_{b^{-1}}K(pab, b, pa)\big)
~\mu_{t(a)}(p).
\end{equation}
By Proposition \ref{2024.04.08} we can write $\mu=\sum\limits_{i, j}\mu_{i, j}$ with $supp(\mu_{i, j})\subseteq U_i\hat\times V_j$. By Corollary \ref{2024.03.31}, there exist 
$1_i, 1_j\in C_c^{\infty}(G)$ with $supp(1_i)\subseteq U_i^{-1}$ and $supp(1_j)\subseteq V_j^{-1}$.

Then we claim the sum of the 4th, 5th, and 6th terms of \eqref{2022.11.3}  equals
\[
\sum_{i, j}\bigg(
\big[d_21_{i}, \mu_{i, j}\star1_{j}\big]
+\big[1_{i}, d_2\mu_{i, j}\star1_{j}\big]\bigg)
(\gamma_0; \gamma_1, \gamma_2),
\]
which is zero in $\Omega_c^{\bullet}(G)_{\rm{Ab}}$.
Observe that
\[\begin{split}
&\big[d_21_{i}, \mu_{i, j}\star1_{j}\big]
+\big[1_{i}, d_2\mu_{i, j}\star1_{j}\big]\\
=&d_21_{i}\star\mu_{i, j}\star1_{j}+
\mu_{i, j}\star1_{j}\star d_21_{i}+
1_{i}\star d_2\mu_{i, j}\star1_{j}-
d_2\mu_{i, j}\star1_{j}\star 1_{i}\\
=&\mu_{i, j}\star1_{j}\star d_21_{i}-
d_2\mu_{i, j}\star1_{j}\star 1_{i}.
\end{split}\]
We then directly compute
\[\begin{split}
&\sum_{i, j}\big(\mu_{i, j}\star1_{j}\star d_21_{i}\big)(\gamma_0; \gamma_1, \gamma_2)\\
=&\sum_{i, j}\sum_{\gamma\gamma'=\gamma_1}(\mu_{i, j}\star1_{j})(\gamma_0; \gamma)d_21_{i}(\gamma'; \gamma_2)
-\sum_{i, j}\sum_{\gamma\gamma'=\gamma_0}(\mu_{i, j}\star1_{j})(\gamma; \gamma')d_21_{i}(\gamma_1; \gamma_2)\\
=&\sum_{i, j}(\mu_{i, j}\star1_{j})(\gamma_0; \gamma_1)1_{i}(\gamma_2)\\
=&\sum_{i, j}\Big(\sum_{\gamma\gamma'=\gamma_1}\mu_{i, j}(\gamma_0; \gamma)1_{j}(\gamma')-
\sum_{\gamma\gamma'=\gamma_0}\mu_{i, j}(\gamma; \gamma')1_{j}(\gamma_1)\Big)
1_{i}(\gamma_2)\\
=&-\chi_{G^{(0)}}(\gamma_0\gamma_1\gamma_2)\mu
(\gamma_2^{-1}; \gamma_1^{-1}).\\
\end{split}\]
Substituting Equation \eqref{2024.06.07} for $\mu$, the above formula is equal to
\[\begin{split}
&-\chi_{G^{(0)}}(\gamma_0\gamma_1\gamma_2)
\int_{Z_{t(\gamma_0)}}h(p)h(p\gamma_2^{-1})
tr\big(A_{\gamma_1}K(p\gamma_0, \gamma_1^{-1}, p\gamma_2^{-1})\big)
~\mu_{t(\gamma_0)}(p)\\
=&-\chi_{G^{(0)}}(\gamma_0\gamma_1\gamma_2)
\int_{Z_{t(\gamma_0)}}h(p)h(p\gamma_2^{-1})
tr\big(A_{\gamma_1}B_{\gamma_2^{-1}}K(p\gamma_2^{-1}\gamma_1^{-1},  \gamma_0, p)\big)
~\mu_{t(\gamma_0)}(p)\\
=&-\chi_{G^{(0)}}(\gamma_0\gamma_1\gamma_2)
\int_{Z_{t(\gamma_0)}}h(p)h(p\gamma_2^{-1})
tr\big(A_{\gamma_2}A_{\gamma_1}K(p\gamma_2^{-1}\gamma_1^{-1},  \gamma_0, p)\big)
~\mu_{t(\gamma_0)}(p)\\
=&-\chi_{G^{(0)}}(\gamma_0\gamma_1\gamma_2)
\int_{Z_{t(\gamma_0)}}h(p)h(p\gamma_2^{-1})
tr\big(A_{\gamma_0^{-1}}K(p\gamma_2^{-1}\gamma_1^{-1},  \gamma_0, p)\big)
~\mu_{t(\gamma_0)}(p),\\
\end{split}\]
which is the fifth term of \eqref{2022.11.3}.
It remains to compute
\[\begin{split}
&\sum_{i, j}\big(-
d_2\mu_{i, j}\star1_{j}\star 1_{i}\big)(\gamma_0; \gamma_1, \gamma_2)\\
=&\sum_{i, j}\Big(
-\sum_{\gamma\gamma'=\gamma_2}\chi_{G^{(0)}}(\gamma_0)\chi_{G^{(0)}}(\gamma_1\gamma_2)\mu_{i, j}(\gamma_1; \gamma)
+\sum_{\gamma\gamma'=\gamma_1}\chi_{G^{(0)}}(\gamma_0)\chi_{G^{(0)}}(\gamma_1\gamma_2)\mu_{i, j}(\gamma; \gamma')\\
&-\chi_{G^{(0)}}(\gamma_0\gamma_1\gamma_2)\mu_{i, j}(\gamma_0; \gamma_1)\Big)\\
=&\chi_{G^{(0)}}(\gamma_0)\chi_{G^{(0)}}(\gamma_1\gamma_2)
\sum_{\gamma\gamma'=\gamma_1}\int_{Z_{t(\gamma_0)}}h(p)h(p\gamma)
tr\big(A_{\gamma'^{-1}}K(p\gamma_1, \gamma', p\gamma)\big)
~\mu_{t(\gamma_0)}(p)\\
&-\chi_{G^{(0)}}(\gamma_0\gamma_1\gamma_2)
\int_{Z_{t(\gamma_0)}}h(p)h(p\gamma_0)
tr\big(A_{\gamma_1^{-1}}K(p\gamma_2^{-1}, \gamma_1, p\gamma_0)\big)
~\mu_{t(\gamma_0)}(p),\\
\end{split}\]
which is the sum of the fourth and sixth terms of \eqref{2022.11.3}.
We have completed the proof of Theorem \ref{2024.03.04}.
\end{proof}
\noindent{\bf Acknowledgment.}~~
The author would like to thank Prof. Bingkwan So for his
helpful advice and guidance.\\
\bibliographystyle{alpha}
\bibliography{ref}
\end{document}